\documentclass[12pt]{amsart}

\usepackage[dvips]{graphicx}
	\usepackage{amsmath,amssymb,amsthm,wrapfig,amsfonts,enumerate,latexsym}

	\newtheorem{dfn}{Definition}[section]
	\newtheorem{thm}[dfn]{Theorem}
	\newtheorem{prop}[dfn]{Proposition}
	\newtheorem{cor}[dfn]{Corollary}
	\newtheorem{lem}[dfn]{Lemma}
	\newtheorem{rem}[dfn]{Remark}
	\newtheorem{ex}[dfn]{Example}
 	\newtheorem{claim}[dfn]{Claim}
	
	\newtheorem{conj}[dfn]{Conjecture}
	
	\newtheorem{ack}{Acknowledgements\!\!}

	\newcounter{yon}
	\numberwithin{equation}{section}

	\def\notin{\not\in}

	\newcommand{\dist}{\mathop{\mathit{d}} \nolimits}
	\newcommand{\id}{\mathop{\mathrm{id}} \nolimits}
	\newcommand{\diam}{\mathop{\mathrm{diam}} \nolimits}

	\newcommand{\pr}{\mathop{\mathrm{pr}} \nolimits}

	\newcommand{\sep}{\mathop{\mathrm{Sep}} \nolimits}
	
	\newcommand{\obs}{\mathop{\mathrm{ObsDiam}}  \nolimits}

	\newcommand{\supp}{\mathop{\mathrm{Supp}}    \nolimits}

	\newcommand{\crad}{\mathop{\mathrm{CRad}}           \nolimits}
	
	\newcommand{\grad}{\mathop{\mathrm{grad}}                   \nolimits}

	\newcommand{\ric}{\mathop{\mathit{Ric}}        \nolimits}

\begin{document}

	\title[Concentration of maps and group action]
    {Concentration of maps and group action}
	\author[Kei Funano]{Kei Funano}
	\address{Mathematical Institute, Tohoku University, Sendai 980-8578, JAPAN}
	\email{sa4m23@math.tohoku.ac.jp}
	\thanks{This work was partially supported by Research Fellowships of
	the Japan Society for the Promotion of Science for Young Scientists.}
	\dedicatory{}
	\date{\today}

	\maketitle


\begin{abstract}In this paper, from the viewpoint of the concentration
 theory of maps, we study a compact group and a L\'{e}vy group action to a
 large class of metric spaces, such as $\mathbb{R}$-trees, doubling
 spaces, metric
 graphs, and Hadamard manifolds. 
 \end{abstract}
	\setlength{\baselineskip}{5mm}

    \section{Introduction}

    Let a compact metric group $G$ acts on a compact metric space
    $X$. In \cite[Theorem 5.1]{mil4}, V. Milman considered a H\"{o}lder
    action (see Section 3.6.2 for the definition) and estimated the
    diameters of orbits from above by words of an isoperimetric
    property of the group $G$ and a covering property of $X$. As he refered in the
    introduction, his idea came from the fixed point theory of a
    L\'{e}vy group action by M. Gromov and Milman in \cite[Theorem
    7.1]{milgro} (see Section 4 for the definition of a L\'{e}vy group). In this paper,
    we consider general continuous actions of a compact metric group and a
    L\'{e}vy group to some concrete
    noncompact metric spaces, such as $\mathbb{R}$-trees, doubling
 spaces, metric
 graphs, and Hadamard manifolds.

 Of isoperimetric inspiring, the L\'{e}vy-Milman concentration theory of maps played an important role in Milman's
    estimation (and also Gromov and Milman's theorem of a L\'{e}vy group
    action). Taking a point $x\in X$, he considered how concentrates the orbit map $G\ni g \to gx\in
    X$ to a constant map. Recent developments of the concentration
    theory of maps by the author (\cite{funano2},
           \cite{funad}, \cite{funano1}), by Gromov (\cite{gromovcat},
    \cite{gromov}), and by M. Ledoux and
    K. Oleszkiewvicz (\cite{ledole}) enable us to estimate how the orbit map
    concentrate to a constant map in the case where $X$ is an
    $\mathbb{R}$-tree, a doubling space, a metric graph, and a Hadamard
    manifold. In stead of considering a H\"{o}lder action and a covering property, we provide an estimate of the diameters
    of orbits of a continuous action of a compact metric group to those metric spaces by
    words of the continuity of the action, an isoperimetric property of
    $G$, and a metric space property of $X$. Our results assert that we
    can measure how the
    action to those metric spaces is closed to the trivial action by the
    above words.

    In the same point of view, we obtain two results of a L\'{e}vy group
    action to the above spaces. A L\'{e}vy group was first introduced and analyzed by
    Gromov and Milman in \cite{milgro}. Gromov and Milman proved that
    every continuous action of a L\'{e}vy group to a compact metric
    space has a fixed point. They also pointed out that the unitary group $U(\ell^2)$ of the separable
     Hilbert space $\ell^2$ with the strong topology is a L\'{e}vy
    group. Many concrete examples of L\'{e}vy groups are known by the
    works of S. Glasner \cite{gla}, H. Furstenberg and B. Weiss
    (unpublished), T. Giordano and V. Pestov \cite{giopes1}, \cite{giopes2}, and
    Pestov \cite{pestov1}, \cite{pestov3}. For examples, groups of measurable maps from the
    standard Lebesgue measure space to compact groups, unitary groups of some
    von Neumann algebras, groups of measure and measure-class preserving
    automorphisms of the standard Lebesgue measure space, full groups of
    amenable equivalence relations, and the isometry groups of the
    universal Urysohn metric spaces are L\'{e}vy groups (see the recent
    monograph \cite{pestov2} for precise). One of our results states
    that there is no non-trivial uniformly continuous action of a L\'{e}vy group to the above
    spaces (Proposition \ref{th3}). We also obtain a generalization of Gromov and Milman's fixed
    point theorem (Proposition \ref{th2}). Both two results are obtained
    by making Gromov and Milman's argument precise.

    The article is organized as follows. In Section $2$, we recall
    basic facts about the concentration theory of maps and prepare for
    the Sections $3$ and $4$. In Section $3$,
    we estimates the diameter of orbits of a compact group action to
    $\mathbb{R}$-trees, doubling spaces, meric graphs, and Hadamard
    manifolds. Section $4$ is devoted to a L\'{e}vy group action to
    those spaces.
     \section{Preliminaries}
\subsection{Concentration function and observable diameter}
In this subsection, we recall some basic facts in the concentration
theory of 1-Lipschitz maps. We recall relationships between an
isoperimetric property of an mm-space (metric measure space) and the concentration theory of $1$-Lipschitz functions. The concentration
theory of $1$-Lipschitz functions was introduced by Milman in his
investigations of asymptotic geometric analysis (\cite{mil1},
\cite{mil2}, \cite{mil3}). While the concentration theory of functions
developed, the concentration theory of maps into general metric spaces
was first studied by Gromov (\cite{gromovcat}, \cite{gromov2},
\cite{gromov}). He established the theory by introducing the observable
diameter in \cite{gromov}. We first recall its definition. 

Let $Y$ be a metric space and $\nu$ a Borel measure on $Y$ such that $m:=\nu(Y)<+\infty$. 
	 We define for any $\kappa >0$
	 \begin{align*}
	  \diam (\nu , m-\kappa):= \inf \{ \diam Y_0 \mid Y_0 \subseteq Y \text{ is a Borel subset such that }\nu(Y_0)\geq m-\kappa\}
	  \end{align*}and call it the \emph{partial diameter} of $\nu$.

      Let $(X,\dist_X)$ be a complete sparable metric space equipped with a
finite Borel measure $\mu_X$ on $X$. Henceforth, we call such a triple
an \emph{mm-space}.
   \begin{dfn}[Observable diameter]\upshape Let $(X,\dist_X,\mu_X)$ be
      an mm-space with $m_X:=\mu_X(X)$ and $Y$ a metric space. For any $\kappa >0$ we
	 define the \emph{observable diameter} of $X$ by 
	 \begin{align*}
	  \obs_Y (X; -\kappa):=
	   \sup \{ \diam (f_{\ast}(\mu_X),m_X-\kappa) \mid f:X\to Y \text{ is a
      }1 \text{{\rm -Lipschitz map}}  \}, 
      \end{align*}where $f_{\ast}(\mu_X)$ stands for the push-forward
      measure of $\mu_X$ by $f$.
      \end{dfn}
      The idea of the observable diameter comes from the quantum and statistical
	mechanics, that is, we think of $\mu_X$ as a state on a configuration
	space $X$ and $f$ is interpreted as an observable.

    Given sequences $\{X_n\}_{n=1}^{\infty}$ of mm-spaces and $\{
    Y_n\}_{n=1}^{\infty}$ of metric spaces, observe that $\lim_{n\to
    \infty}\obs_{Y_n}(X_n;-\kappa)=0$ for any $\kappa >0$ if and only if
    for any sequence $\{ f_n:X_n \to Y_n\}_{n=1}^{\infty}$ of
    $1$-Lipschitz maps there exists a sequence $\{
    m_{f_n}\}_{n=1}^{\infty}$ of points such that $m_{f_n}\in Y_n$ and
    \begin{align*}
     \lim_{n\to \infty}\mu_{X_n}(\{ x_n \in X_n \mid
     \dist_{Y_n}(f_n(x_n),m_{f_n})\geq \varepsilon\})=0
     \end{align*}for any $\varepsilon>0$. A sequence
     $\{X_n\}_{n=1}^{\infty} $ of
     mm-spaces is said to be a \emph{L\'{e}vy family} if $\lim_{n\to
     \infty}\obs_{\mathbb{R}}(X_n;-\kappa)=0$ for any $\kappa>0$. The
     concept of L\'{e}vy families was first introduced in \cite{milgro}.

  For an mm-space $X$ with $\mu_X(X)=1$, we define the \emph{concentration function}
        $\alpha_X:(0,+\infty)\to \mathbb{R}$ as the supremum of
        $\mu_X(X\setminus A_{+r})$, where $A$ runs over all Borel subsets
        of $X$ with $\mu_X(A)\geq 1/2$ and $A_{+r}$ is an open
        $r$-neighbourhood of $A$. This function describes an isoperimetric
        feature of the space $X$.

        We shall consider each closed Riemannian manifold as an mm-space
        equipped with the volume measure normalized to have the total
        volume $1$.

         \begin{ex}\label{exl1}\upshape Let $M$ be a closed Riemannian manifold such
         that $\ric_M \geq \widetilde{\kappa}_1>0$. By virtue of the L\'{e}vy-Gromov
         isoperimetric inequality, we obtain $\alpha_M(r)\leq
         e^{-\widetilde{\kappa}_1 r^2/2}$ (see \cite[Section 1.2, Remark
         2]{milgro} or \cite[Theorem 2.4]{ledoux}). Since
         $\ric_{SO(n)}\geq (n-1)/4$, we have $\alpha_{SO(n)}(r)\leq
          e^{-(n-1)r^2/8}$ for example.
        \end{ex}

        \begin{ex}\label{exl2}\upshape Let $M$ be a closed Riemannian
         manifold. We denote by $\lambda_1(M)$ the non-zero first
         eigenvalue of the Laplacian on $M$. Then, for any $r >0$, we have $\alpha_M(r)\leq
         e^{-\sqrt{\lambda_1(M)}r/3}$ (see \cite[Theorem 4.1]{milgro} or
         \cite[Theorem 3.1]{ledoux}). Since the $n$-dimensional torus
         $\mathbb{T}^n:=\mathbb{S}^1 \times \mathbb{S}^1 \times \cdots
         \times \mathbb{S}^1$ satisfies
         $\lambda_1(\mathbb{T}^n)=\lambda_1(\mathbb{S}^1)=1$, we obtain
         $\alpha_{\mathbb{T}^n}(r)\leq e^{-r/3}$ for example.
         \end{ex}

         Let $X$ be an mm-space and $f:X\to \mathbb{R}$ a Borel
         measurable function. A number $m_f\in \mathbb{R}$ is called a
         \emph{median} of $f$ if it satisfies that $f_{\ast}(\mu_X)((-\infty,
         m_f]) \geq m_X /2$ and $f_{\ast}(\mu_X)([m_f ,+\infty))\geq
         m_X/2$. We remark that $m_f$ does exist, but it is not unique
         in general.

         Relationships between the concentration function and the
         observable diameter are the following:
        \begin{lem}[{cf.~\cite[Section 1.3]{ledoux}}]\label{vel1}Let $X$ be an
         mm-space with $\mu_X(X)=1$. Then, for any $1$-Lipschitz
         function $f:X\to \mathbb{R}$ and $\varepsilon >0$, we have
         \begin{align*}
          \mu_X (\{ x\in X \mid |f(x)-m_f|\geq \varepsilon \})\leq
          2\alpha_X(\varepsilon). 
          \end{align*}
         \end{lem}
 \begin{lem}[{cf.~\cite[Section 1.3]{ledoux}}]\label{vel2}Let $X$ be an mm-space with $\mu_X(X)=1$. Assume that a function $\alpha:(0,+\infty)\to
          \mathbb{R}$ satisfies that
          \begin{align*}
             \mu_X (\{ x\in X \mid |f(x)-m_f|\geq \varepsilon \})\leq \alpha(\varepsilon)
           \end{align*}for any $1$-Lipschitz function $f:X\to
          \mathbb{R}$. Then, we have $\alpha_X(\varepsilon) \leq
          \alpha(\varepsilon)$. 
          \end{lem}

           By Lemmas \ref{vel1} and \ref{vel2}, we obtain the following corollary:
           \begin{cor}[{\cite[Section 1.3]{ledoux}}]A sequence $\{X_n\}_{n=1}^{\infty}$ of mm-spaces is a
      L\'{e}vy family if and only if $\lim_{n\to
      \infty}\alpha_{X_n}(r)=0$ for any $r>0$.
      \end{cor}

         Combining Lemma \ref{vel1} with Examples \ref{exl1} and \ref{exl2}, we
         obtain the following corollaries: 
         \begin{cor}Let $M$ be a closed Riemannian manifold such that
          $\ric_M \geq \widetilde{\kappa}_1>0$. Then, for any $\kappa
          >0$, we have
          \begin{align*}
           \obs_{\mathbb{R}}(M;-\kappa)\leq 2\sqrt{\frac{2\log \big(\frac{2}{\kappa}\big)}{\widetilde{\kappa}_1}}.
           \end{align*}In particular, we have
          \begin{align*}
           \obs_{\mathbb{R}}(SO(n);-\kappa)\leq 4\sqrt{\frac{2\log\big(\frac{2}{\kappa}\big)}{n-1}}.
           \end{align*}
          \end{cor}

          \begin{cor}Let $M$ be a closed Riemannian manifold. Then, for
           any $\kappa >0$, we have
           \begin{align*}
            \obs_{\mathbb{R}}(M;-\kappa)\leq \frac{6 \log \big( \frac{2}{\kappa}\big)}{\sqrt{\lambda_1(M)}}.
            \end{align*}
           In particular, we have
           \begin{align*}
            \obs_{\mathbb{R}}(\mathbb{T}^n;-\kappa)\leq 6 \log \Big(\frac{2}{\kappa}\Big).
            \end{align*}
           \end{cor}
      
\subsection{Concentration and separation}

In this section, we recall the notion of the separation distance for an
     mm-space which was introduced in \cite{gromov}. We review
     relationships between the observable diameter and the separation
     distance. The separation distance plays an important role
     throughout this paper. 

     Let $X$ be an
      mm-space. For $\kappa_1,  \kappa_2\geq 0$, we
      define the \emph{separation distance} $\sep
      (X;\kappa_1,\kappa_2)= \sep (\mu_X;\kappa_1,\kappa_2)$ of $X$ as the supremum of
      the distance $\dist_X(A,B)$, where $A$ and $B$ are Borel subsets
      of $X$ satisfying that $\mu_X(A)\geq \kappa_1$ and $\mu_X(B)\geq \kappa_2$.

     Relationships between the observable diameter and the separation distance
     are followings. We refer to \cite[Subsection 2.2]{funad} for precise proofs.
     	\begin{lem}[{cf.~\cite[Section $3\frac{1}{2}.33$]{gromov}}]\label{noranoraneko}Let $X$ be an mm-space and $\kappa,\kappa' >0$ with $\kappa > \kappa'$. Then we have
	 \begin{align*}
	  \obs_{\mathbb{R}} (X ;-\kappa')\geq \sep (X;\kappa,\kappa).
	  \end{align*}
         \end{lem}

         	\begin{rem}\upshape In {\cite[Section $3\frac{1}{2}.33$]{gromov}}, Lemma \ref{noranoraneko} is stated as $\kappa =\kappa'$, but that is not true in general. For example, let
	 $X:=\{ x_1 , x_2\}$, $\dist_X (x_1,x_2):=1$, and $\mu_X (\{ x_1\})=\mu_X (\{  x_2   \}):= 1/2$. Putting $\kappa =\kappa'=1/2$, we have
	 $\obs_{\mathbb{R}} (X;-1/2)=0$ and $\sep
	 (X;1/2,1/2)=1$.
	 \end{rem}

   \begin{lem}[cf.~{\cite[Section
    $3\frac{1}{2}.33$]{gromov}}]\label{l2.1.2}Let $\nu$ be a Borel measure on
    $\mathbb{R}$ with $m:=\nu(\mathbb{R})<+\infty$. Then, for any $\kappa >0$ we have
	 \begin{align*}
	  \diam (\nu, m-2\kappa)\leq \sep (\nu; \kappa, \kappa).
      \end{align*}
      In particular, for any $\kappa >0$ we have
      \begin{align*}
        \obs_{\mathbb{R}}(X;-2\kappa)\leq \sep (X; \kappa, \kappa).
       \end{align*}
    \end{lem}

    \begin{cor}[{cf.~\cite[Section $3\frac{1}{2}.33$]{gromov}}]\label{c2.1.1}A sequence $\{ X_n\}_{n=1}^{\infty}$ of mm-spaces is a
     L\'{e}vy family if and only if $\lim_{n\to \infty}\sep
     (X_n;\kappa,\kappa) =0$ for any $\kappa >0$. 
     \end{cor}

\subsection{Compact metric group action and diameter of a measure}

     Let a compact metric group $G$ continuously acts on a metric space $X$. For each $\eta >0$, we define a (possibly infinite) number $\rho(\eta)= \rho^{(G,X)}(\eta)$
as the supremum of $\dist_X(gx,gy)$ for all $g\in G$ and $x,y\in X$
with $\dist_X(x,y)\leq \eta$. Given a point $x\in X$, we indicate by
$f_x:G\to X$ the orbit map of $x$, that is, $f_x(g):=gx$ for any $g\in
G$. For the Haar measure $\mu_G$ on $G$ normalized as $\mu_G(G)=1$, we
put $\nu_{G,x}:=(f_x)_{\ast}(\mu_G)$. 
      \begin{prop}\label{p3.1}Assume that $\nu_{G,x}(B_X(y,\delta))>1/2$ for some
       $y\in X$ and $\delta >0$. 
       Then, we have
       \begin{align}\label{s3.1}
        \dist_X(y,gy)\leq \delta + \rho(\delta)
        \end{align}for any $g\in G$. Moreover, there exists a point $x_0 \in Gx$ such that
       \begin{align}\label{s3.2}
      \dist_X (x_0,gx_0) \leq \min \{
        2\delta + \rho(2\delta),2\delta + 2\rho(\delta)   \}
        \end{align}for any $g\in G$.
       \begin{proof}Taking any $g\in G$, we first prove (\ref{s3.1}). Since $gB_X(y,\delta)\subseteq
        B_X(gy, \rho (\delta))$ and the measure $\nu_{G,x}$ is
        $G$-invariant, from the assumption, we have
        \begin{align*}
         \nu_{G,x}(B_X(gy,\rho(\delta)))\geq \nu_{G,x}(gB_X(y,\delta))=\nu_{G,x}(B_X(y,\delta))>1/2.
         \end{align*}Combining this with $\nu_{G,x}(B_X(y,\eta))>1/2$,
        we get $\nu_{G,x}(B_X(y,\delta)\cap B_X(gy,\rho(\delta)))>0$, which
        implies (\ref{s3.1}).

        We next prove (\ref{s3.2}). Since the orbit $Gx$ is compact,
        the support of the measure $\nu_{G,x}$ is included in
        $Gx$. Hence, there exists a point $x_0 \in B_X(y,\delta)\cap
        Gx$. Let $g\in G$. Since $\nu_{G,x}(B_X(x_0,2\delta))\geq
        \nu_{G,x}(B_X(x_0,2\delta))>0$, by using (\ref{s3.1}), we obtain
        $\dist_X(x_0,gx_0)\leq 2\delta +\rho(2\delta)$. We also have
        \begin{align*}
         \dist_X(x_0,gx_0)\leq \ &\dist_X(x_0,y)+ \dist_X(y,gy)+
         \dist_X(gy, g x_0)\\
         \leq \ & \delta + (\delta + \rho(\delta))+ \rho(\delta)\\
         = \ & 2\delta+ 2\rho(\delta),
         \end{align*}which implies (\ref{s3.2}). This completes the proof.
        \end{proof}
       \end{prop}

        \begin{prop}\label{p3.2}Assume that $\nu_{G,x}(A)>1/2$ for some Borel subset $A\subseteq X$. Then, there exists a point $x_0\in Gx$ such that
         \begin{align*}
          \dist_X(x_0,gx_0)\leq \diam A+\rho(\diam A)
          \end{align*}for any $g\in G$.
         \begin{proof}Since $A\cap Gx\neq \emptyset$, the claim follows
          from the same argument in the proof of Proposition \ref{p3.1}.
          \end{proof}
         \end{prop}

         For any $\eta>0$, we put $\rho(+\eta):=\lim_{\eta'\downarrow
         \eta}\omega_x(\eta')$. 
         \begin{cor}\label{c3.1}There exists a point
          $z_x \in Gx$ such that
          \begin{align*}
           \dist_X (z_x , gz_x) \leq \lim_{\kappa \uparrow
          1/2}\diam (\nu_{G,x},1-\kappa)  +\rho\big(+ \lim_{\kappa \uparrow
          1/2}\diam (\nu_{G,x},1-\kappa) \big)
           \end{align*}for any $g\in G$.
          \end{cor}

          For any $\eta>0$, we define a (possibly infinite) number
          $\omega_x(\eta)=\omega_x^{(G,X)}(\eta)$ as the supremum of
          $\dist_X(gx, g'x)$ for all $g,g'\in G$ with $\dist_G(g,g')\leq
          \eta$. 
          \begin{lem}\label{l3.1}For any $\kappa_1,\kappa_2 >0$, we have
           \begin{align*}
            \sep (\nu_{G,x};\kappa_1,\kappa_2)\leq
            \omega_x(+\sep(G;\kappa_1, \kappa_2)).
            \end{align*}
           \begin{proof}Let $A$ and $B$ be two Borel subsets such that
            $\nu_{G,x}(A)\geq \kappa_1$ and $\nu_{G,x}(B)\geq
            \kappa_2$. Since $\mu_G((f_x)^{-1}(A))\geq \kappa_1$ and
            $\mu_G((f_x)^{-1}(B))\geq \kappa_2$, we have
            $\dist_G((f_x)^{-1}(A), (f_x)^{-1}(B))\leq \sep
            (G;\kappa_1,\kappa_2)$. Thus, from the definition of
            $\omega_x$, we obtain $\dist_X(A,B)\leq
            \omega_x(+\sep(G;\kappa_1,\kappa_2))$. This completes the proof.
           \end{proof}
           \end{lem}

           \begin{cor}[{cf.~\cite[Section 5.2]{milgro}}]\label{c3.2}Assume that a sequence $\{ G_n\}_{n=1}^{\infty}$ of
            compact metric groups is a L\'{e}vy family and each $G_n$ acts on a metric
            space $X$. Assume also that there exist a sequence
            $\{x_n\}_{n=1}^{\infty}$ of points in $X$ and a function
            $\omega:(0,+\infty) \to [0,+\infty]$ such that $\lim_{\eta
            \to 0}\omega (\eta)=0$ and $\omega^{(G_n,X)}_{x_n}(\eta)\leq
            \omega(\eta)$ for any $n\in \mathbb{N}$ and $\eta >0$. Then,
            the sequence $\{ (X,\dist_X,
            \nu_{G_n,x_n})\}_{n=1}^{\infty}$ of mm-spaces is a L\'{e}vy family.
            \end{cor}

 \section{Estimates of the diameters of orbits}

 Throughout this section, we always assume that a compact metric group $G$ continuously acts on a metric space
 $X$. We shall consider the group $G$ as an mm-space $(G,\dist_G,\mu_G )$, where
 $\mu_G$ is the Haar measure on $G$ normalized as $\mu_G(G)=1$. In this section, motivated by the work of Milman \cite{mil4}, we
 shall estimate the diameters of orbits $Gx$ from above for concrete
 metric spaces $X$ by words of the continuity of the action, an
 isoperimetric property of $G$, and a metric space property of $X$. For this purpose, we use the notation
 $\rho=\rho^{(G,X)}$ and $\omega_x=\omega_x^{(G,X)}$ defined in
 Subsection 2.3. 
 We first consider the case where the orbit map $f_x:G\ni
 g\mapsto gx\in X$ for some $x\in X$ is a $1$-Lipschitz map. In this case,
 applying Corollary \ref{c3.1}, we
 obtain the following:
\begin{prop}For any $\kappa\in (0,1/2)$, there exists a point
 $z_{\kappa}\in X$ such that
 \begin{align*}
  \dist_X(z_{\kappa},gz_{\kappa})\leq \obs_X(G;-\kappa)+ \rho (\obs_X(G;-\kappa))
  \end{align*}for any $g\in G$.
 \end{prop}

 \subsection{Case of Euclidean spaces}
 In this subsection, we consider the case where the metric space $X$ is the Euclidean space
 $\mathbb{R}^k$. Let $\pr_i:\mathbb{R}^k \ni x= (x_i)_{i=1}^{k}\mapsto x_i
   \in \mathbb{R}$ be the projection.
 \begin{prop}[{cf.~\cite[Section $3\frac{1}{2}.32$]{gromov}}]\label{p4.1.1}For any finite Borel measure
  $\nu$ on $\mathbb{R}^k$ with $m:=\nu(\mathbb{R}^k)$, we have
  \begin{align*}
   \diam (\nu,m-\kappa)\leq \sqrt{k}\max_{1\leq i\leq k} \diam
   \Big((\pr_i)_{\ast}(\nu), m-\frac{\kappa}{k}\Big).
   \end{align*}
  \end{prop}

  Applying Corollary \ref{c2.1.1} to Proposition \ref{p4.1.1}, we obtain the following corollary:
  \begin{cor}[{cf.~\cite[Section $3\frac{1}{2}.32$]{gromov}}]For any
   L\'{e}vy family $\{X_n\}_{n=1}^{\infty}$ and any $\kappa>0$, we have 
   \begin{align*}
    \lim_{n\to \infty}\obs_{\mathbb{R}^k}(X_n;-\kappa)=0.
    \end{align*}
   \end{cor}
 \begin{prop}\label{p4.1.2}Assume that a compact metric group $G$
  continuously acts
  on the Euclidean space $\mathbb{R}^k$ and put $r:=\lim_{\kappa
  \uparrow 1/(4k)}\sep (G;\kappa,\kappa)$. Then, for any $x\in \mathbb{R}^k$, there
  exists a point $z_{x}\in Gx$ such that
  \begin{align}\label{s4.1.1}
   \dist_{\mathbb{R}^k}(z_{x},g z_{x})\leq \sqrt{k} \omega_x (+ r )+
   \rho(+\sqrt{k}\omega_x (+ r ) )
   \end{align}for any $g\in G$.
  \begin{proof}Combining Lemma \ref{l3.1} with Proposition \ref{p4.1.1}, we get
   \begin{align*}
    \diam (\nu_{G,x},1-\kappa) \leq \ &\sqrt{k}\max_{1\leq i\leq k} \diam
    \Big((\pr_i)_{\ast}(\nu_{G,x}),1- \frac{\kappa}{k}\Big)\\
    \leq \ & \sqrt{k}\max_{1\leq i\leq k} \sep
    \Big((\pr_i)_{\ast}(\nu_{G,x});\frac{\kappa}{2k},\frac{\kappa}{2k}\Big)\\
    \leq \ & \sqrt{k}\sep
    \Big(\nu_{G,x};\frac{\kappa}{2k},\frac{\kappa}{2k}\Big)\\
    \leq \ & \sqrt{k} \omega_x \Big(+\sep
    \Big(G;\frac{\kappa}{2k},\frac{\kappa}{2k}\Big) \Big).
    \end{align*}Applying this to Corollary \ref{c3.1}, we obtain
   (\ref{s4.1.1}). This completes the proof.
   \end{proof}
  \end{prop}

   \subsection{Case of compact metric spaces}

In this subsection, we treat the case where the metric space $X$ is a compact metric
space $K$. For any $\delta >0$, we denote by $N_K(\delta)$ the minimum
number of Borel subsets of diameter at most $\delta$ which cover $K$.

  \begin{prop}[{cf.~\cite[Section $3\frac{1}{2}.34$]{gromov}}]\label{p4.2.1}For any
   $\delta,\kappa>0$ and any finite Borel measure $\nu$ on $K$ with $m:=\nu(K)$, we have
   \begin{align*}
    \diam (\nu,m-\kappa)\leq \sep \Big(\nu;\frac{\kappa}{N_K(\delta)},
    \frac{\kappa}{N_K(\delta)} \Big) +2\delta.
    \end{align*}
   \end{prop}

   \begin{cor}[{cf.~\cite[Section $3\frac{1}{2}.34$]{gromov}}]Let $\{
    X_n\}_{n=1}^{\infty}$ be a L\'{e}vy family and $K$ a compact metric
    space. Then, for any $\kappa >0$, we have
    \begin{align*}
     \lim_{n\to \infty}\obs_{K}(X_n;-\kappa)=0.
     \end{align*}
    \end{cor}
   By virtue of Proposotion \ref{p4.2.1}, the same proof of Proposition \ref{p4.1.2} yields the following proposition:
  \begin{prop}\label{p4.2.2}Assume that a compact metric group $G$
   continuously acts on a compact metric space $K$ and put $r_{x,\delta}:= \omega_x(+\lim_{\kappa \uparrow 1/(2N_K(\delta))}\sep
    (G;\kappa,\kappa))+2\delta$ for $x\in K$ and $\delta>0$.
     Then, there
  exists a point $z_{x,\delta}\in Gx$ such that
   \begin{align*}
    \dist_K(z_{x,\delta},gz_{x,\delta})\leq \ &
       r_{x,\delta}+ \rho(+r_{x,\delta})
    \end{align*}for any $g\in G$.
   \end{prop}

   Proposition \ref{p4.2.2} generalizes Milman's result \cite[Theorem 5.1]{mil4}.

   \subsection{Case of $\mathbb{R}$-trees}
   In this subsection, we consider the case where the metric space $X$ is an $\mathbb{R}$-tree $T$. For this purpose, we first recall
    some standard terminologies in metric geometry. Let $(X,\dist_X)$ be
    a metric space. A rectifiable curve $\gamma:[0,1]\to X$ is called a
    \emph{geodesic} if its arclength coincides with
    the distance $\dist_X(\gamma(0),\gamma(1))$ and it has a constant speed,
    i.e., parameterized proportionally to the arc length. We say that
    $(X,\dist_X)$ is a \emph{geodesic space} if any two points in $X$ are joined
    by a geodesic between them.

    A complete metric space $T$ is called an 
    \emph{$\mathbb{R}$-tree} if it has the following properties:
    \begin{itemize}
		   \item[$(1)$]Any two points in $T$ are connected by a unique unit
                       speed geodesic.
		   \item[$(2)$]The image of every simple path in $T$ is the
                       image of a geodesic.
	 \end{itemize}

To answer Gromov's exercise in \cite[Section $3\frac{1}{2}.32$]{gromov}, the author proved the
following theorem:
   \begin{thm}[{cf.~\cite[Proposition 5.1]{funano2}}]\label{t4.3.1}For any $\kappa>0$
    and finite Borel measure $\nu$ on $T$ with $m:=\nu(T)$, we have
    \begin{align*}
     \nu \Big(B_T \Big(x_{\nu},\sep\Big(\nu;
     \frac{\kappa}{2},\frac{m}{3}   \Big) \Big) \Big)\geq 1-\kappa.
     \end{align*}
   \end{thm}

   \begin{cor}[{cf.~\cite[Theorem 1.1]{funano2}}]Let
    $\{X_n\}_{n=1}^{\infty}$ be a L\'{e}vy family and $T$ an
    $\mathbb{R}$-tree. Then, for any $\kappa >0$, we have
    \begin{align*}
     \lim_{n\to \infty}\obs_T(X_n;-\kappa)=0.
     \end{align*}
    \end{cor}
By Proposition \ref{p3.1} and Theorem \ref{t4.3.1}, the following
proposition follows from the same proof of Proposition
\ref{p4.1.2}.
  \begin{prop}Assume that a compact metric group $G$ continuously acts
   on an $\mathbb{R}$-tree $T$. Then, for any $x\in T$ and $\kappa\in (0,1/4)$, there
   exists a point $ z_{x,\kappa}\in T$ such that
   \begin{align*}
    \dist_T(z_{x,\kappa},gz_{x,\kappa} )\leq
    \omega_x \Big( +\sep\Big(G;\kappa,\frac{1}{3}\Big)\Big)  +   \rho\Big( \omega_x \Big(
    +\sep\Big(G;\kappa,\frac{1}{3}\Big)\Big)\Big)
    \end{align*}for any $g\in G$. Put $r:=\lim_{\kappa\uparrow 1/4}\sep
   (G;\kappa,\kappa)$. Then, there also exists a point
   $z_{x} \in Gx$ such that
   \begin{align*}
    \dist_T (z_{x} ,gz_{x})\leq \ &\min \{2\omega_x (+r) +
    \rho(+2\omega_x(+r)), 2\omega_x (+r)+
    2\rho(\omega_x (+r) \}
    \end{align*}for any $g\in G$.
   \end{prop}

   \subsection{Case of doubling spaces}Throughout this subsection,
    we consider the case where the metric space $X$ is a doubling space.
    A complete metric space $X$ is called a \emph{doubling space} if
    there exist $R_1>0$ and a function $D=D_X:(0,R_1]\to (0,+\infty)$
    satisfying the following condition:
    Every closed ball with radius $2r_1\leq 2R_1$ is covered by at most
    $D(r_1)$ closed
    balls with radius $r_1$. This condition is equivalent to the
    following condition: There exists a function
    $C=C_X=C(r_1,r_2):(0, 2R_{1}]\times (0,2R_1]\to (0,+\infty)$ such that
    for every $(r_1,r_2)\in (0,2R_1]\times (0,2R_1]$, every $r_1$-separated subset in any closed ball in
    $X$ with
    radius $r_2$ contains at most $C(r_1,r_2)$ elements. For example, a
    complete Riemannian manifold with Ricci curvature bounded from below
    is a doubling space (see the proof of Corollary \ref{c4.4.3}).

    Although the proof of
    the following theorem is the same analogue to \cite[Theorem
    1.3]{funad}, we give it for completeness.

   \begin{thm}\label{t4.4.1}Let $X$ be a doubling space and $\nu$ a finite Borel measure on
    $X$ with $m:= \nu(X)$. Assume that a positive number $r_0$ satisfies
    \begin{align*}
     r_0> \max \Big\{ \sep \Big(\nu;\kappa, \frac{m}{C(r_0,5r_0)}\Big),
     \sep \Big(  \nu;   \frac{m-\kappa}{3},\frac{m-\kappa}{3}
     \Big), \sep \Big(\nu; \frac{m-\kappa}{3},        \kappa  \Big)\Big\}
     \end{align*}for some $\kappa >0$. Then there exists a point $x_{0}\in X$ such that $\nu (
    B_X(x_{0},3r_0))\geq m-\kappa$.
    \begin{proof}Take a maximal $r_0$-separated set $\{\xi_{\alpha}
     \}_{\alpha\in \mathcal{A}}$ of $X$. From the doubling property of
     $X$, there exists $\alpha_0\in
     \mathcal{A}$ such that 
     \begin{align*}
      k:= \# \{   \beta\in \mathcal{A} \mid
      \xi_{\beta}\in B_X(\xi_{\alpha_0}, 5r_0)  \} =  \max_{\alpha \in \mathcal{A}} \#\{   \beta\in \mathcal{A} \mid
      \xi_{\beta}\in B_X(\xi_{\alpha}, 5r_0)  \}\leq C(r_0,5r_0).
      \end{align*}Putting $\{\beta_1, \beta_2 ,\cdots , \beta_k\}:=\{   \beta\in \mathcal{A} \mid
      \xi_{\beta}\in B_X(\xi_{\alpha_0}, 5r_0)  \}  $, we take a subset
     $J_1 \subseteq \{ \xi_{\alpha}\}_{\alpha\in \mathcal{A}}$ which is
     maximal with respect to the properties that $J_1$ is $5r_0$-separated
     and $\xi_{\beta_1}\in J_1$, $\xi_{\beta_2}\notin J_1$, $\cdots$, $\xi_{\beta_k}
     \notin J_1$. We then take $J_2 \subseteq \{ \xi_{\alpha}\}_{\alpha \in
     \mathcal{A}} \setminus J_1$ which is maximal with respect to the
     properties that $J_2$ is $5r_0$-separated and $ \xi_{\beta_2}\in
     J_2$, $\xi_{\beta_3}\notin J_2$, $\cdots$, $\xi_{\beta_k}\notin
     J_2$. In the same way, we pick $J_3 \subseteq \{
     \xi_{\alpha}\}_{\alpha \in \mathcal{A}}\setminus (J_1 \cup J_2)$,
     $\cdots$, $J_k\subseteq \{ \xi_{\alpha}\}_{\alpha\in \mathcal{A}}
     \setminus (J_1 \cup J_2 \cup \cdots \cup J_{k-1})$. We then have
     \begin{claim}\label{cl4.4.1}$\{\xi_{\alpha} \}_{\alpha \in \mathcal{A}} = J_1 \cup J_2 \cup
      \cdots \cup J_k$.
      \begin{proof}Suppose that $\xi_{\alpha}\notin J_1 \cup J_2 \cup
       \cdots \cup J_k$ for some $\alpha \in \mathcal{A}$. Since each
       $J_i$ is maximal, there exists $\xi_{\alpha_i} \in J_i$ such that
       $\dist_X(\xi_{\alpha}, \xi_{\alpha_i})<5r_0$ and
       $\xi_{\alpha}\neq \xi_{\alpha_i}$. We therefore obtain
       \begin{align*}
        k+1 \leq \# \{   \xi_{\alpha}, \xi_{\alpha_1}, \xi_{\alpha_2},
        \cdots , \xi_{\alpha_k}           \}\leq \# \{ \beta \in
        \mathcal{A}\mid \xi_{\beta}\in B_X(\xi_{\alpha},5r_0)\}\leq k,
        \end{align*}which is a contradiction. This completes the proof
       of the claim.
       \end{proof}
      \end{claim}
     By Claim \ref{cl4.4.1}, we have $X= \bigcup_{i=1}^k\bigcup_{\xi_{\alpha}\in
     J_i}B_X(\xi_{\alpha}, r_0)$. Hence there exists $i$, $1\leq i\leq
     k$ such that
     \begin{align*}
      \nu\Big(\bigcup_{\xi_{\alpha}\in
      J_i}B_X(\xi_{\alpha},r_0)\Big)\geq \frac{m}{k}\geq \frac{m}{C(r_0,5r_0)}.
      \end{align*}We then have
     \begin{claim}\label{cl4.4.2}
      \begin{align*}
       \nu \Big(  \bigcup_{\xi_{\alpha}\in J_i}B_X(\xi_{\alpha},2r_0)
       \Big)\geq m-\kappa.
      \end{align*}
      \begin{proof}Supposing that $\nu (  \bigcup_{\xi_{\alpha}\in J_i}B_X(\xi_{\alpha},2r_0)
       )< m-\kappa$, from the assumption of $r_0$, we have
       \begin{align*}
        r_0 \leq \dist_X\Big(X\setminus \bigcup_{\xi_{\alpha}\in
        J_i}B_X(\xi_{\alpha}, 2r_0), \bigcup_{\xi_{\alpha}\in
        J_i}B_X(\xi_{\alpha},r_0)\Big)\leq \sep \Big(\nu;\kappa,\frac{m}{C(r_0,5r_0)}\Big)<r_0.
        \end{align*}This is a contradiction. This completes the proof of
       the claim.
       \end{proof}
      \end{claim}
     \begin{claim}\label{cl4.4.3}There exists $\xi_{\gamma}\in J_i$ such that
      $ \nu(B_X(\xi_{\gamma}, 2r_0))\geq (m-\kappa)/3$.
      \begin{proof}Suppose that $\nu(B_X(\xi_{\alpha}, 2r_0))<
       (m-\kappa)/3$ for any $\xi_{\alpha}\in J_i$. Then, by Claim \ref{cl4.4.2},
       there exists $J_i'\subseteq J_i$ such that
       \begin{align*}
        \frac{m-\kappa}{3}\leq \nu \Big(\bigcup_{\xi_{\alpha}\in
        J_i'}B_X(\xi_{\alpha}, 2r_0)\Big)< \frac{2(m-\kappa)}{3}.
        \end{align*}Thus, putting $J_i'':= J_i \setminus J_i'$, from the
       assumption of $r_0$, we get
       \begin{align*}
        r_0 \leq \dist_X \Big(   \bigcup_{\xi_{\alpha}\in
        J_i'}B_X(\xi_{\alpha}, 2r_0),  \bigcup_{\xi_{\alpha}\in
        J_i''}B_X(\xi_{\alpha}, 2r_0)
        \Big)\leq \sep \Big(\nu;\frac{m-\kappa}{3}, \frac{m-\kappa}{3}\Big)<r_0.
        \end{align*}This is a contradiction. This completes the proof of
       the claim. 
       \end{proof}
      \end{claim}Combining Claim \ref{cl4.4.3} with the same method of the proof of Claim \ref{cl4.4.2}, we
     finally obtain $\nu(B_X(\xi_{\gamma}, 3r_0))\geq 1-\kappa$. This
     completes the proof of the theorem.
     \end{proof}
    \end{thm}

    By Corollary \ref{c2.1.1} and Theorem \ref{t4.4.1}, we get the
    following corollary:

    \begin{cor}[{cf.~\cite[Theorem 1.3]{funad}}]\label{c4.4.1}Let
     $\{X_n\}_{n=1}^{\infty}$ be a L\'{e}vy family and $X$ a doubling
     space. Then, for any $\kappa >0$, we have
     \begin{align*}
      \lim_{n\to \infty}\obs_X(X_n;-\kappa)=0.
      \end{align*}
     \end{cor}

     Applying Theorem \ref{t4.4.1} to Proposition \ref{p3.1}, we obtain the
     following proposition: 
    \begin{prop}\label{p4.4.1}Let a compact metric group $G$ continuously acts on a doubling space $X$. Assume that a positive number $r_0$ satisfies 
     \begin{align*}
      r_0 >\ & \max \Big\{   \omega_{x}\Big(+\sep \Big(\nu;\kappa, \frac{1}{C(r_0,5r_0)}\Big)\Big),
     \omega_x\Big(+\sep \Big(  \nu;   \frac{1-\kappa}{3},\frac{1-\kappa}{3}
     \Big)\Big), \\ \ &
      \hspace{9cm}\omega_x \Big(+\sep \Big(\nu; \frac{1-\kappa}{3},        \kappa  \Big)   \Big)    \Big\}
      \end{align*}for some $\kappa\in (0,1/2)$. Then there exists a point $z_{x,\kappa}\in X$ such that
     \begin{align*}
      \dist_X(z_{x,\kappa},gz_{x,\kappa})\leq 3r_0 +\rho(3r_0)
      \end{align*}for any $g\in G$. Moreover, there exists a point $z_{x,\kappa}'\in Gx$ such that
     \begin{align*}
      \dist_X (z_{x,\kappa}',gz_{x,\kappa}')\leq \min \{  6r_0 + \rho(6r_0), 6r_0 + 2\rho(3r_0)            \}
      \end{align*}for any $g\in G$.
     \end{prop}

     We next consider the case where the function $D=D_X:(0,+\infty
     )\to (0,+\infty)$ is a constant function. This is equivalent to the
     following condition: The function $C=C_X:(0,+\infty)\times
     (0,+\infty)\to (0,+\infty)$ satisfies that $C(\alpha r, \alpha s)= C(r,s)$ for any
     $r,s, \alpha>0$. We call such a metric space a
     \emph{large scale doubling space}. 

     By Theorem \ref{t4.4.1}, we obtain the following corollary:
     \begin{cor}\label{c4.4.2}Let $X$ be a large scale doubling space
      and $\nu$ be a finite Borel measure on $X$ with
      $m:=\nu(X)$ and put
      \begin{align*}
       r_{\kappa}:=\max \Big\{ \sep \Big(\nu;\kappa, \frac{m}{C(1,5)}\Big),
     \sep \Big(  \nu;   \frac{m-\kappa}{3},\frac{m-\kappa}{3}
     \Big), \sep \Big(\nu; \frac{m-\kappa}{3},        \kappa  \Big)\Big\}
       \end{align*}for $\kappa>0$. Then, there exists a point
      $x_{\kappa}\in X$ such that $\nu(B_X(x_{\kappa},3r_{\kappa}))\geq m-\kappa$.
      \end{cor}
      Applying Corollary \ref{c4.4.2} to Proposition \ref{p3.1}, we
      obtain the following proposition:
   \begin{prop}\label{p4.4.2}Assume that a compact metric group $G$ continuously acts on a
    large scale doubling space $X$. Put
    \begin{align*}
      r_{x,\kappa} := \ &\max\Big\{
     \omega_x\Big(+\sep\Big(G;\kappa, \frac{1}{C(1,5)}\Big)\Big),
     \omega_x\Big( +\sep\Big(G;\frac{1-\kappa}{3},
     \frac{1-\kappa}{3}\Big) \Big), \\ & \hspace{8cm}\omega_x\Big(+ \sep\Big(
     G;\frac{1-\kappa}{3},\kappa                   \Big)\Big) 
     \Big\}
     \end{align*}for $x\in X$ and $\kappa>0$. Then, for any $\kappa \in (0,1/2)$,
    there exists a point $z_{x,\kappa}\in X$ such that
    \begin{align*}
     \dist_X (z_{x,\kappa},gz_{x,\kappa}) \leq \ &
     3r_{x,\kappa}+ \rho (3r_{x,\kappa} ) 
     \end{align*}for any $g\in G$. There also exists a point
    $z_{x,\kappa}' \in Gx$ such that
    \begin{align*}
     \dist_X (z_{x,\kappa}',gz_{x,\kappa}')\leq \min \{
      6r_{x,\kappa}+\rho(6r_{x,\kappa}), 6r_{x,\kappa}+2\rho(3r_{x,\kappa})        \}
     \end{align*}for any $g\in G$.
    \end{prop}

    Assume that a complete metric space $X$ has a doubling measure
    $\nu_X$, that is, $\nu_X$ is a (not only finite) Borel measure on $X$
    having the following properties: $X=\supp \nu_X$ and there exists a
    constant $C=C(X)>0$ such that
    \begin{align}
     \nu_X(B_X(x,2r))\leq C \nu_X(B_X(x,r))
     \end{align}for any $x\in X$ and $r>0$. For example, by virtue of
     the Bishop-Gromov volume comparison theorem, the volume measure of an $n$-dimensional complete Riemannian
     manifold $M$ with nonnegative Ricci curvature is a doubling measure
     with $C(M)=2^n$.

    \begin{lem}[{cf.~\cite[Lemma 2.1]{funad}}]\label{l4.4.1}Let $(X,\nu_X)$ be a complete
     metric space with a doubling measure $\nu_X$. Then, for any $0< r_1\leq r_2$ and $x,y\in
     X$ with $x\in B_X(y,r_2)$, we have
     \begin{align*}
      \frac{\nu_X(B_X(x,r_1))}{\nu_X(B_X(y,r_2))}\geq
      \frac{1}{C^2}\Big(\frac{r_1}{r_2}\Big)^{\log_2
      C}=C^{\log_2 \frac{r_1}{r_2}-2}.
      \end{align*}
     \end{lem}

     \begin{cor}\label{c4.4.3}The space $(X,\nu_X)$ is a large scale doubling
      space with $C_X(r_1,r_2)\leq C^{2+\log_2 \{ (r_1+2r_2)/r_1\}}$. In particular, we have 
      $C_X(1,5)\leq C^{2+\log_2 11}$.
      \begin{proof}Given any $x\in X$ and $r_1,r_2>0$
       with $r_2\geq r_1$, we let $\{ \xi_{\alpha}\}_{\alpha \in \mathcal{A}}
       \subseteq B_X(x,r_2)$ be an arbitrary $r_1$-separated set. Note that
       closed balls $B_X(\xi_{\alpha}, 2^{-1}r_1 -\varepsilon)$ are mutually
       dijoint for any $\varepsilon>0$. We hence have
       \begin{align*}
        \nu_X(B_X(x, 2^{-1}r_1 +r_2))\geq \ &\nu_X\Big(\bigcup_{\alpha
        \in \mathcal{A}}B_X(\xi_{\alpha}, 2^{-1}r_1-\varepsilon) \Big)\\
        =\ & \sum_{\alpha \in \mathcal{A}} \nu_X (B_X(\xi_{\alpha},
        2^{-1}r_1-\varepsilon))\\
        \geq \ & \nu_X(B_X(\xi_{\alpha_0}, 2^{-1}r_1-\varepsilon))\# \mathcal{A},
        \end{align*}where $\nu_X(B_X(\xi_{\alpha_0}, 2^{-1}r_1-\varepsilon))=
       \min_{\alpha\in \mathcal{A}}\nu_X(B_X(\xi_{\alpha},
       2^{-1}r_1-\varepsilon))$. Applying this to Lemma \ref{l4.4.1}, we obtain
      \begin{align*}
       \# \mathcal{A} \leq \frac{\nu_X(B_X(x, 2^{-1}r_1+r_2))}{\nu_X
       (B_X(\xi_{\alpha_0}, 2^{-1}r_1-\varepsilon))}\leq C^{2+ \log_2
       \{ (r_1+2r_2)/ (r_1-2\varepsilon)\}}.
       \end{align*}This completes the proof.
       \end{proof}
     \end{cor}

     Combining Corollary \ref{c4.4.2} with Corollary \ref{c4.4.3}, we obtain the
     following corollary:
     \begin{cor}\label{c4.4.4}Let $\nu$ be a finite Borel measure on $(X,\nu_X)$ with
      $m:=\nu(X)$. Put 
     \begin{align*}
    r_{\kappa}:= \max \Big\{ \sep
      (\nu;\kappa,C^{-2-\log_2 11}), \sep \Big(
      \nu;\frac{m-\kappa}{3}, \frac{m-\kappa}{3} \Big), \sep\Big(  \nu;
      \frac{m-\kappa}{3} , \kappa\Big) \Big\}
      \end{align*}for $\kappa>0$. Then, there exists a point $x_{\kappa}\in X$ such that
     $\nu(B_X (x_{\kappa}, 3r_{\kappa}))\geq 1-\kappa$. In particular,
     we have $\diam (\nu,m-\kappa)\leq 6r_{\kappa}$.
    \end{cor}

    By using Corollary \ref{c4.4.4}, we obtain the following propostion:
    \begin{prop}\label{p4.4.3}Assume that a compact metric group $G$ continuously acts on
     $(X,\nu_X)$. Put
     \begin{align*}
     r_{x,\kappa}:= &\max \Big\{ \omega_x (+\sep
      (G;\kappa,C^{-2-\log_2 11})), \omega_x\Big(+\sep \Big(
      G;\frac{1-\kappa}{3}, \frac{1-\kappa}{3} \Big)\Big), \\
      &\hspace{10cm} \omega_x\Big(+\sep\Big(  G;
      \frac{1-\kappa}{3} , \kappa\Big)\Big) \Big\}
     \end{align*}for $x\in X$ and $\kappa>0$. Then, for any $\kappa \in (0,1/2)$,
    there exists a point $z_{x,\kappa}\in X$ such that
    \begin{align*}
     \dist_X (z_{x,\kappa},gz_{x,\kappa}) \leq \ &
     3r_{x,\kappa}  + \rho(3r_{x,\kappa})
     \end{align*}for any $g\in G$. There also exists a point
    $z_{x,\kappa}' \in Gx$ such that
    \begin{align*}
     \dist_X (z_{x,\kappa}',gz_{x,\kappa}')\leq \min \{
     6r_{x,\kappa}+\rho(6r_{x,\kappa}), 6r_{x,\kappa}+2\rho(3r_{x,\kappa})        \}
     \end{align*}for any $g\in G$.
     \end{prop}

     \begin{cor}Assume that a compact metric group $G$ continuously acts on an
      $n$-dimensional complete Riemannian manifold $M$ with nonnegative
      Ricci curvature. Put
      \begin{align*}
        r_{\kappa}:= &\max \Big\{ \omega_x (+\sep
      (G;\kappa,2^{-(2+\log_2 11)n})), \omega_x\Big(+\sep \Big(
      G;\frac{1-\kappa}{3}, \frac{1-\kappa}{3} \Big)\Big), \\
      &\hspace{10cm} \omega_x\Big(+\sep\Big(  G;
      \frac{1-\kappa}{3} , \kappa\Big)\Big) \Big\}
       \end{align*}for $x\in M$ and $\kappa>0$. Then, for any $x\in
      M$ and $\kappa \in (0,1/2)$,
    there exists a point $z_{x,\kappa}\in M$ such that
    \begin{align*}
     \dist_M (z_{x,\kappa},gz_{x,\kappa}) \leq \ &
     3r_{x,\kappa}  + \rho(3r_{x,\kappa})
     \end{align*}for any $g\in G$. There also exists a point
    $z_{x,\kappa}' \in Gx$ such that
    \begin{align*}
     \dist_M (z_{x,\kappa}',gz_{x,\kappa}')\leq \min \{
     6r_{x,\kappa}+\rho(6r_{x,\kappa}), 6r_{x,\kappa}+2\rho(3r_{x,\kappa})  \}
     \end{align*}for any $g\in G$.
      \end{cor}

    \subsection{Case of metric graphs}In this subsection, we treat the
     case where $X$ is a metric graph. Let $\Gamma=(V,E)$ be a (possibly
      infinite) undirected connected combinatorial graph, that is,
      $\Gamma$ is a $1$-dimensional cell complex with the set $V$ of vertices and
      the set $E$ of edges. We allow
      the graph $\Gamma$ to have multiple edges and loops. For vertices
      $v,w\in V$ which are endpoints of an edge, we assign a
      positive number $a_{v,w}$ such that $a_{\Gamma}:=\inf_{v'\neq
      w'}a_{v'w'}>0$.  Every edge is
      identified with a bounded closed interval or a circle in $\mathbb{R}^2$
      with lengh $a_{vw}$, where $v$ and $w$ are endpoints of the
      edge. We then define the distance between two points in $\Gamma$
      to be the infimum of the length of paths joining them. The graph
      $\Gamma$  together with such a distance function is called a
      \emph{metric graph}.

      \begin{lem}\label{l4.5.1}Let $(C,\dist_C)$ be a circle in $\mathbb{R}^2$ with
       the Riemannian distance function $\dist_C$ and $\nu $ a
       finite Borel measure on $C$ with $m:=\nu(C)$. Then, for any
       $\kappa>0$, we have
       \begin{align*}
        \diam(\nu,m-\kappa)\leq \frac{\pi}{\sqrt{2}} \sep \Big(\nu;
        \frac{\kappa}{4}, \frac{\kappa}{4} \Big)
        \end{align*}
       \begin{proof}Note that
        \begin{align*}
         \dist_{\mathbb{R}^2}(x,y)\leq \dist_C(x,y) \leq \frac{\pi}{2} \dist_{\mathbb{R}^2}(x,y)
         \end{align*}for any $x,y\in C$. Denoting by $\pr_i:\mathbb{R}^2
        \ni (x_1,x_2)\mapsto x_i\in \mathbb{R}$ the projection, by using
        Lemma \ref{l2.1.2}, we
        therefore obtain
        \begin{align*}
         \diam (\nu,m-\kappa)=\ & \diam (\nu|_{(C,\dist_C)},m-\kappa)\\
         \leq \ &\frac{\pi}{2}\diam
         (\nu|_{(C,\dist_{\mathbb{R}^2})},m-\kappa)\\
         \leq \ &\frac{\pi}{\sqrt{2}}\max_{i=1,2} \diam\Big(
         (\pr_i)_{\ast}(
         \nu|_{(C,\dist_{\mathbb{R}^2})}),m-\frac{\kappa}{2}
         \Big)\\
         \leq \ &\frac{\pi}{\sqrt{2}}\max_{i=1,2} \sep \Big( (\pr_i)_{\ast}(
         \nu|_{(C,\dist_{\mathbb{R}^2})});\frac{\kappa}{4},\frac{\kappa}{4}
         \Big)\\
          \leq \ &\frac{\pi}{\sqrt{2}} \sep \Big( 
         \nu|_{(C,\dist_{\mathbb{R}^2})};\frac{\kappa}{4},\frac{\kappa}{4}
         \Big)\\
         \leq \ & \frac{\pi}{\sqrt{2}} \sep \Big( 
         \nu;\frac{\kappa}{4},\frac{\kappa}{4}
         \Big).
         \end{align*}This completes the proof.
        \end{proof}
       \end{lem}

       For every edge $e\in E$ and $r>0$, we put $e_{-r}:=\{ x\in e \mid
       \dist_{\Gamma}(e,v)>r \text{ and } \dist_{\Gamma}(e,w)>r\}$,
       where $v$ and $w$ are endpoints of the edge $e$. 
    \begin{thm}\label{t4.5.1}Let $\nu$ be a finite Borel measure on a metric graph $\Gamma$ with
     $m:=\nu(\Gamma)$. Assume that postive numbers $a,\kappa,\kappa'$
     satisfy that $\kappa'< \kappa$, $a<a_{\Gamma}$, and 
     \begin{align*}
      \max \Big\{  2\sep \Big(
      \nu;\frac{\kappa}{3},\frac{\kappa}{3}\Big), 4\sep \Big(
      \nu;\frac{m-\kappa}{3}, \kappa'\Big)                    \Big\}<a
      \end{align*}Then, we have
     \begin{align}\label{s4.5.1}
      \diam (\nu,m-\kappa)\leq \max \Big\{  \frac{a}{2}+2\sep \Big(\nu;
      \frac{\kappa}{3},\kappa\Big), \frac{\pi}{\sqrt{2}}\sep \Big(\nu;\frac{\kappa-\kappa'}{4},\frac{\kappa-\kappa'}{4}\Big)  \Big\}.
      \end{align}
     \begin{proof}We first consider the case of $\nu (\bigcup_{v\in
      V}B_X(v,a/4))\geq \kappa$. Since $\sep
      (\nu;\kappa/3,\kappa/3)<a/2$, as in the proof of Claim \ref{cl4.4.3}, there
      exists a vertex $v\in V$ such that $\nu (B_X(v,a/4))\geq \kappa
      /3$. We thus obtain $\nu (B_X(v,a/4
      +\sep(\nu;\kappa/3,\kappa/3)))\geq m-\kappa$, which implies (\ref{s4.5.1}).

      We consider the other case that $\nu (X\setminus \bigcup_{v\in
      V}B_X(v,a/4)) >m-\kappa$. By the same method of Claim \ref{cl4.4.3}, either the
      following (1) or (2) holds:

      (1) There exists an edge $e\in E$ such that $e$ is not a loop and $\nu (e_{-a/4})\geq
      (m-\kappa)/3$.

      (2) There exists a loop $\ell\in E$ with $\nu (\ell_{-a/4})\geq
      (m-\kappa)/3$.

      If (1) holds, combining the same proof of Claim \ref{cl4.4.2} with
      $\sep(\nu;\kappa/3,\kappa')<a/4$, we then have $\nu (e)\geq
      m-\kappa'$. We therefore obtain
      \begin{align*}
       \diam (\nu,m-\kappa)\leq \ &\diam (\nu|_{e},m-\kappa)\\ =\ &\diam
       (\nu|_{e}, \nu(e)-(\nu(e)-m+\kappa))\\
       \leq \ &\sep \Big(\nu|_{e};
       \frac{\nu(e)-m+\kappa}{2},\frac{\nu(e)-m+\kappa}{2}\Big)\\
       \leq \ & \sep \Big(\nu; \frac{\kappa-\kappa'}{2}, \frac{\kappa-\kappa'}{2}\Big).
       \end{align*}If (2) holds, by Claim \ref{cl4.4.2} and
      $\sep(\nu;\kappa/3,\kappa')< a/4$, we then get
      $\nu(\ell)\geq m-\kappa'$. Applying Lemma \ref{l4.5.1}, we therefore obtain
      \begin{align*}
       \diam (\nu,m-\kappa)\leq \ &\diam (\nu|_{\ell},m-\kappa)\\
       = \ & \diam (\nu|_{\ell}, \nu(\ell)- (\nu(\ell)-m+\kappa))\\
       \leq \ & \frac{\pi}{\sqrt{2}} \sep \Big( \nu|_{\ell};
       \frac{\nu(\ell)-m+\kappa}{4},  \frac{\nu(\ell)-m+\kappa}{4}\Big)\\
       \leq \ & \frac{\pi}{\sqrt{2}} \sep \Big( \nu|_{\ell};
       \frac{\kappa -\kappa'}{4},  \frac{\kappa -\kappa'}{4}\Big)\\
         \leq \ & \frac{\pi}{\sqrt{2}} \sep \Big( \nu ;
       \frac{\kappa -\kappa'}{4},  \frac{\kappa -\kappa'}{4}\Big).
       \end{align*}This completes the proof of the theorem.
      \end{proof}
     \end{thm}

     \begin{cor}Let $\{X_n\}_{n=1}^{\infty}$ be a L\'{e}vy family and
      $\Gamma$ a metric graph. Then, for any $\kappa >0$, we have
      \begin{align*}
       \lim_{n\to \infty}\obs_{\Gamma}(X_n;-\kappa)=0.
       \end{align*}
      \end{cor}

      By virtue of Theorem \ref{t4.5.1}, we obtain the following:
     \begin{prop}Assume that a compact metric group $G$ continuously
      acts on a metric graph $\Gamma$. We also assume that postive
      numbers $a,\kappa,\kappa'$ and a point $x\in X$ satisfy that $\kappa'< \kappa$, $a<a_{\Gamma}$, and 
      \begin{align*}
       \max\Big\{2\omega_x\Big(+\sep\Big(G;\frac{\kappa}{3},
       \frac{\kappa}{3}\Big)\Big), 4\omega_x \Big(+\sep
       \Big(G;\frac{1-\kappa}{3}, \kappa'\Big)\Big)\Big\}<a.
       \end{align*}Put
      \begin{align*}
       s_{x,\kappa,\kappa'}:= \max \Big\{ \frac{a}{2} + 2\omega_x
       \Big(+\sep \Big(G;\frac{\kappa}{3},\kappa  \Big)\Big),
       \frac{\pi}{\sqrt{2}}\omega_x \Big(+\sep \Big( G;\frac{\kappa-\kappa'}{4},\frac{\kappa-\kappa'}{4}\Big)\Big)\Big\}.
       \end{align*}Then, there exists a point $z_{x,\kappa,\kappa'}\in Gx$ such that
      \begin{align*}
       \dist_X( z_{x,\kappa,\kappa'}, gz_{x,\kappa,\kappa'})\leq
       s_{x,\kappa,\kappa'}+ \rho(s_{x,\kappa,\kappa'})
       \end{align*}for any $g\in G$.
      \end{prop}
\subsection{Case of Hadamard manifolds}
In this subsection, we consider the case where $X$ is a Hadamard
manifold $N$, i.e., a complete simply connected Riemannian manifold with
nonpositive sectional curvature. The following theorem was obtained in \cite[Theorem 1.3]{funano1}.
\begin{thm}Let $\{X_n\}_{n=1}^{\infty}$ be a L\'{e}vy family and $N$ a
 Hadamard manifold. Then, for any $\kappa >0$, we have
 \begin{align*}
  \lim_{n\to \infty}\obs_N (X_n;-\kappa)=0.
  \end{align*}
 \end{thm}
   \subsubsection{Central radius}

Let $N$ be a Hadamard manifold. For
a finite Borel measure on $N$ with compact support, we indicate the center of
mass of the measure $\nu$ by $c(\nu)$. Given any $\kappa>0$, putting
$m:= \nu(N)$, we define the \emph{central radius} $\crad(\nu,m-\kappa)$
of $\nu$ as the infimum of $\rho>0$ such that $\nu(B_N(c(\nu),\rho))\geq
m-\kappa$.

  \begin{prop}[{cf.~\cite[Proposition 5.4]{sturm}}]\label{p4.6.1.1}For a finite Borel measure $\nu$ on
   $\mathbb{R}^k$ with the compact support, we have
\begin{align*}
 c(\nu)=\frac{1}{\nu(\mathbb{R}^k)}\int_{\mathbb{R}^k}x d\nu(x).
 \end{align*}
   \end{prop}

   \begin{prop}[{cf.~\cite[Proposition 5.10]{sturm}}]\label{p4.6.1.2}Let $N$ be a
    Hadamard manifold and $nu$ a finite Borel measure on $N$ with the
    compact support. Then, $x=c(\nu)$ if and only if
    \begin{align*}
     \int_N \exp_x^{-1}(y)d\nu(y)=0.
     \end{align*}In particular, identifying the tangent space of $N$ at
    the point $c(\nu)$ with the Euclidean space of the same dimension of
    $N$, we have $c((\exp_{c(\nu)}^{-1})_{\ast}(\nu))=0$.
   \end{prop}
Proposition \ref{p3.1} directly implies the following corollary:
          \begin{cor}\label{c4.6.1.1}Assume that a compact metric group acts on a
           Hadamard manifold $N$ and put $r_x:=\lim_{\kappa \uparrow
           1/2}\crad(\nu_{G,x},1-\kappa)$ for $x\in X$. Then, we have
        \begin{align*}
         \dist_X(c(\nu_{G,x}), gc(\nu_{G,x}))\leq r_x+ \rho(+r_x)
         \end{align*}for any $g\in G$. Moreover, there exists
        a point $z_{x}\in Gx$ such that
        \begin{align*}
         \dist_X(z_{x},gz_{x}) \leq \ &\min \{
         2r_x+\rho(+2r_x), 
         2r_x+2\rho
         (+r_x)
         \} 
         \end{align*}for any $g\in G$.
        \end{cor}

        \subsubsection{H\"{o}lder actions}

        In this subsubsection, we consider a H\"{o}lder action of a
        compact Lie group to a Hadamard manifold.

        Let a compact Lie group $G$ acts on a Hadamard manifold $N$. We
        shall consider the case where $\omega_x(\eta)\leq C_1
        \eta^{\alpha}$ holds for some $x\in N$ and $C_1,\alpha>0$. 

        Combining Gromov's observation in \cite[Section
        13]{gromovcat} with one in \cite[Section
        $3\frac{1}{2}.41$]{gromov}, we obtain the following theorem:
        \begin{thm}\label{t4.6.2.1}Let $M$ be a compact Riemannian
         manifold and $N$ be a Hadamard manifold. Assume that a
         continuous map $f:M\to N$ satisfies that
         \begin{align*}
          \dist_N(f(x),f(y))\leq C_1 \dist_M (x,y)^{\alpha}
          \end{align*}for some $C_1>0$, $\alpha> 1$, and all $x,y\in
         M$. Then, the map $f:M\to N$ is a constant map.
         \begin{proof}Put $\mathbb{E}(f):=c(f_{\ast}(\mu_M))$. We shall
          prove that $\supp f_{\ast}(\mu_X)= \{ \mathbb{E}(f)\}$, which
          implies the theorem. Suppose that $\supp f_{\ast}(\mu_X)
          \neq \{ \mathbb{E}(f)\}$. We identify the tangent space of $N$
          at $\mathbb{E}(f)$ with the Euclidean space $\mathbb{R}^k$,
          where $k$ is the dimension of $N$. According to the hinge
          theorem (see \cite[Chapter \Roman{yon}, Remark 2.6]{sakai}), the map $\exp_{\mathbb{E}(f)}^{-1}:N\to \mathbb{R}^k$
          is $1$-Lipschitz. Since the map $\exp^{-1}_{\mathbb{E}(f)}$ is
          isometric on rays issuing from $\mathbb{E}(f)$ and $\supp
          f_{\ast}(\mu_M)\neq \{ \mathbb{E}(f)\}$, we have
          \begin{align*}
           \int_M |(\exp_{\mathbb{E} (f)}^{-1} \circ f)(x)|^2 d\mu_M(x)=
           \int_M \dist_N(f(x),\mathbb{E}(f))^2 d\mu_M(x)>0.
           \end{align*}Denoting by $((\exp_{\mathbb{E} (f)}^{-1} \circ
          f)(x))_i$ the $i$-th component of $(\exp_{\mathbb{E} (f)}^{-1}
          \circ f)(x)$, we hence see that there exists $i_0$ such that
          \begin{align*}
           \int_M |((\exp_{\mathbb{E} (f)}^{-1} \circ f)(x))_{i_0}|^2d\mu_M(x)>0.
           \end{align*}Putting $\varphi:= (\exp_{\mathbb{E} (f)}^{-1}
          \circ f)_{i_0}$, we observe that
          \begin{align*}
           \| \grad_x \varphi \|=\limsup_{y \to
           x}\frac{|\varphi(y)-\varphi (x)|}{\dist_M(y,x)}\leq
           \limsup_{y\to x} \frac{C_1\dist_{M}(y,x)^{\alpha}}{\dist_M(y,x)}= 0
           \end{align*}and the function $\varphi$ has mean zero by
          Proposition \ref{p4.6.1.2}. We therefore obtain
          \begin{align*}
           0<\lambda_1(M)= \inf \frac{\int_M \| \grad_x g \|^2 d\mu_M
           (x)}{\int_M g(x)^2 d\mu_M(x)} \leq \frac{\int_M \| \grad_x
           \varphi\|^2 d\mu_M(x)}{\int_M \varphi(x)^2 d\mu_M(x)}=0,
           \end{align*}where the infimum is taken over all nontrivial Lipschitz
          maps $g:M\to \mathbb{R}$ with mean zero. This is a
          contradiction. This completes the proof.
          \end{proof}
         \end{thm}

    \begin{cor}\label{c4.6.2.1}Assume that a compact Lie group $G$ continuously acts on
     a Hadamard manifold $N$. We also assume that there exists a point
     $x\in X$ such that the condition $\omega_x(\eta)\leq C_1
     \eta^{\alpha}$ holds for some $\alpha>1$. Then, the point $x$ is a
     fixed point.
     \end{cor}

     Assume that a compact metric group $G$ contnuously acts on a Hadamard
     manifold $N$. In view of Corollary \ref{c4.6.2.1}, we shall
     consider the case of $0<\alpha\leq 1$.

        We assume that a compact metric group $G$ satisfies that
        \begin{align}\label{s4.6.2.1}
        \alpha_G(r)\leq C_2 e^{-C_3 r^{\beta}} \text{ for some }C_2,C_3, \beta>0.
         \end{align}See Examples \ref{exl1} and \ref{exl2} for examples.

         Let a compact metric group continuously acts
          on a metric space $X$. For any $r>0$ and $x\in X$, we define
          $\omega_x^{-1}(r)$ as the infimum of $\dist_G(g,g')$, where
          $g$ and $g'$ run over all elements in $G$ such that
          $\dist_X(gx,g'x)\geq r$. 
         \begin{lem}\label{l4.6.2.1}Assume that a compact metric group continuously acts
          on a metric space $X$. Then, for any $x\in X$, we have
          \begin{align*}
           \alpha_{(X,\nu_{G,x})}(r)\leq \alpha_G(\omega_x^{-1}(r)).
           \end{align*}
          \begin{proof}Let $A\subseteq X$ be any Borel subset such that
           $\nu_{G,x}(A)\geq 1/2$. From the difinition of $\omega_x^{-1}(r)$, we get
          \begin{align*}
           \{ g\in G \mid gx\in A\}_{+\omega_x^{-1}(r)}\subseteq \{ g\in
           G \mid gx \in A_{+r}\}.
           \end{align*}Since $\mu_G(\{ g\in G \mid gx\in A\})\geq 1/2$, we hence obtain
           \begin{align*}
            \nu_{G,x}(X\setminus A_{+r})\leq \mu_G(G\setminus \{ g\in G
            \mid gx \in A\}_{+\omega_x^{-1}(r)})\leq \alpha_G(\omega_x^{-1}(r)).
            \end{align*}This completes the proof.
           \end{proof}
          \end{lem}

         \begin{lem}\label{l4.6.2.2}Let a compact metric group $G$ continuously
          acts on a metric space $X$. Assume that a point $x\in X$
          satisfies the following H\"{o}lder condition:
\begin{align}\label{s4.6.2.2}
 \omega_x (\eta)\leq C_1 \eta^{\alpha} \text{ holds for some }C_1>0 \text{ and }0< \alpha\leq 1.
 \end{align}We also assume that the group $G$ satisfies the condition
          (\ref{s4.6.2.1}). Then, we have
\begin{align*}
 \alpha_{(N,\nu_{G,x})}(r)\leq C_2 e^{-C_1^{-\beta/\alpha} C_3
          r^{\beta /\alpha}}.
 \end{align*}
          \begin{proof}By the assumption (\ref{s4.6.2.2}), $\dist_{X}(gx,g'x)>
           C_1s^{\alpha}$ implies that $\dist_G(g,g')>s$, that is,
           $\dist_X(gx,g'x)\geq r$ yields that $\dist_G(g,g')\geq
           (r/C_1)^{1/\alpha}$. We hence get $\omega_x^{-1}(r)\geq
           (r/C_1)^{1/\alpha}$. 
           By using this and Lemma \ref{l4.6.2.1}, we obtain
           \begin{align*}
            \alpha_{(X,\nu_{G,x})}(r)\leq \alpha_G(\omega_x^{-1}(r))\leq
            \alpha_G ((r/C_1 )^{1/\alpha})\leq C_2
            e^{-C_1^{-\beta/\alpha}C_3 r^{\beta /\alpha}}.
            \end{align*}This completes the proof.
           \end{proof}
          \end{lem}

We denote by $\gamma_k$ the standard Gaussian measure on $\mathbb{R}^k$
with density $(2\pi)^{-k/2}e^{-|x|^2 /2}$. For any $p\geq 0$, we put
\begin{align*}
 M_p:=\int_{\mathbb{R}}|s|^p d\gamma_1(s)=2^{p/2}\pi^{-1/2}\Gamma \Big(\frac{p+1}{2}\Big).
 \end{align*}The same proof of \cite[Theorem 1]{ledole} implies the
 following theorem:
\begin{thm}[{cf.~\cite[Theorem 1]{ledole}}]\label{t4.6.2.2}Assume that an mm-space $X$
 satisfies that $\alpha_X(r)\leq C_1e^{-C_2 r^p}$ for some $C_1,C_2>0$
 and some $p\geq 1$. Then, for any $1$-Lipschitz function $f:X\to
 \mathbb{R}^k$ with mean zero, we have
 \begin{align*}
  \int_X|f(x)|^p d\mu_X(x)\leq \frac{C}{C_2 M_p}\int_{\mathbb{R}^k}|y|^p
  d\gamma_k(y) =  \frac{C}{C_2 M_p}\cdot
  \frac{2^{p/2}\Gamma(\frac{p+k}{2})}{\Gamma (\frac{k}{2})} \approx 
\frac{Ck^{p/2}}{C_2},
  \end{align*}where $C$ is a constant depending only on $p$ and $C_1$. 
 \end{thm}

   \begin{thm}\label{t4.6.2.3}Let a compact metric group $G$ continuously acts on a
    $k$-dimensional Hadamard manifold $N$. Assume that a point $x\in
    N$ satisfies the H\"{o}lder condition (\ref{s4.6.2.2}). We also assume that the
    group $G$ satisfies (\ref{s4.6.2.1}) and $\alpha\leq \beta$. Then, there exists a point
    $z_{x}\in Gx$ such that
    \begin{align}\label{s4.6.2.3}
     \diam (G z_{x})\leq \frac{C C_1 k^{1/2}}{(C_3)^{\alpha/\beta}}+ \rho \Big( \frac{C C_1 k^{1/2}}{(C_3)^{\alpha/\beta}}\Big),
     \end{align}where $C$ is a constant depending
    only on $\alpha / \beta$ and $C_1$.
    \begin{proof}To apply Corollary \ref{c4.6.1.1}, we shall estimate $\crad
     (\nu_{G,x},1-\kappa)$ for $0<\kappa <1/2$ from the above. Putting
     $z:=c(\nu_{G,x})$, as in the proof of Theorem \ref{t4.6.2.1}, we identify the tangent space of $N$ at $z$ with
     the Euclidean space $\mathbb{R}^k$. Since the map $\exp_z^{-1}:N\to
     \mathbb{R}^k$ is a $1$-Lipschitz map, by virtue of Lemma \ref{l4.6.2.2} and Theorem \ref{t4.6.2.2}, we
     have
     \begin{align*}
      \int_N \dist_N (y,z)^{\beta/ \alpha}d\nu_{G,x}(y) = \int_N
      |(\exp_z^{-1})(y)|^{\beta /\alpha }d\nu_{G,x}(y)\leq
      \frac{C C_1^{\beta /\alpha}k^{\beta/(2\alpha)}}{C_3},
      \end{align*}where $C$ is a constant depending only on $C_2$ and
     $\beta /\alpha$. Combining this inequality with the Chebyshev
     inequality, we hence get
     \begin{align*}
      \crad(\nu_{G,x},1-\kappa)\leq \frac{CC_1k^{1/2}}{(C_3 \kappa)^{\alpha /\beta}}
      \end{align*}for any $0<\kappa$. Applying Corollary \ref{c4.6.1.1}, we
     therefore obtain (\ref{s4.6.2.3}). This completes the proof. 
     \end{proof}
    \end{thm}

    \subsubsection{Cases of finite groups}

    In this subsubsection, we shall consider the case where $G$ is a
    finite group. Let $G$ be a finite group and $S\subseteq G\setminus
    \{e_G\}$ be a symmetric set of generators of $G$. We denote by
    $\Gamma (G,S)$ the \emph{Cayley graph} of $G$ with respect to
    $S$. For such $S$, we shall consider the group $G$ as a metric group with
    respect to the Cayley graph distance function.

    Let $\Gamma =(V,E)$ be a simple finite graph, where \emph{simple}
    means that there is at most one edge joining two vertices and no
    loops from a vertex to itself. The discrete Laplacian
    $\triangle_{\Gamma}$ act on functions $f$ on $V$ as follows
\begin{align*}
 \triangle_{\Gamma}f(x):= \sum_{y \sim  x}(f(x)-f(y)),
 \end{align*}where $x \sim y$ means that $x$ and $y$ are connected by an
 edge. We denote by $\lambda_1(\Gamma)$ the non-zero first eigenvalue of
 the Laplacian $\triangle_{\Gamma}$. 

     As Theorem \ref{t4.6.2.1}, Gromov's observation in \cite[Section
     13]{gromovcat} together with one in \cite[Section
     $3\frac{1}{2}.41$]{gromov} imply the following lemma:
     \begin{lem}\label{l4.6.3.1}Let $S\subseteq G\setminus \{e_G\}$ be a symmetric set of generators of a finite group $G$ and assume
     that the group $G$ continuously acts on a $k$-dimensional Hadamard
     manifold $N$. Then, for any $x\in N$ and $\kappa>0$, we have
      \begin{align*}
       \crad(\nu_{G,x},1-\kappa)\leq \omega_x(1)
       \Big(\frac{k\#S}{2\kappa\lambda_1(\Gamma (G,S) )}\Big)^{1/2}.
       \end{align*}
      \begin{proof}Suppose that
       \begin{align}\label{s4.6.3.1}
        r:=\crad(\nu_{G,x},1-\kappa)> \omega_x(1)
       \Big(\frac{k\#S}{2\kappa\lambda_1(\Gamma (G,S) )}\Big)^{1/2}.
        \end{align}As in the proof of Theorem \ref{t4.6.2.1}, we identify the
       tangent space of $N$ at $z:=c(\nu_{G,x})$ with the Euclidean space
       $\mathbb{R}^k$. By the Chebyshev inequality, we get
       \begin{align*}
        \int_{G} |(\exp_z^{-1} \circ f^x)(g) |^2d \mu_G(g)= \int_G
        \dist_N(f^x(g),z)^2 d\mu_G(g)\geq \kappa r^2.
        \end{align*}Hence, there exists $i_0$ such that
       \begin{align}\label{s4.6.3.2}
        \int_G ((\exp_z^{-1} \circ f^x)(g))_{i_0}^2d \mu_G(g)\geq
        \frac{\kappa r^2}{k}.
        \end{align}
       Putting $\varphi:= (\exp_z^{-1} \circ f^x)_{i_0}$, by (\ref{s4.6.3.1})
       and (\ref{s4.6.3.2}), we obtain
       \begin{align*}
        \lambda_1(\Gamma (G,S))=\ & \inf\frac{\sum_{g,g'\in G;g\sim g'}
        (f(g)- f(g'))^2}{2\sum_{g\in G} f(g)^2}\\
        \leq \ &
        \frac{\sum_{g,g'\in G;g\sim g'}(\varphi(g)-\varphi(g'))^2}{2\sum_{g\in
        G} \varphi (g)^2}\\
        \leq \ & \frac{\sum_{g,g'\in G;g\sim g'}\dist_N(f^x(g),f^x(g'))^2}{2\sum_{g\in
        G} \varphi (g)^2}\\
        \leq \ & \frac{ \#G \#S \cdot \omega_x(1)^2 }{\#G\int_G \varphi(g)^2
        d\mu_G(g)}\\
        = \ & \frac{\omega_x(1)^2\#S}{\int_{G}\varphi(g)^2 d\mu_G(g)}\\
        \leq \ & \frac{\omega_x(1)^2 k \#S}{\kappa r^2}\\
        < \ & \lambda_1(\Gamma (G,S)),
        \end{align*}where the infimum is taken over all nontrivial
       functions $f:G\to \mathbb{R}$ such that $\sum_{g\in G}f(g)=0$. This is a contradiction. This completes the proof.
       \end{proof}
      \end{lem}

      Applying Lemma \ref{l4.6.3.1} to Corollary \ref{c4.6.1.1}, we obtain the
      following theorem:
    \begin{thm}Let $S\subseteq G\setminus \{e_G\}$ be a symmetric set of generators of a finite group $G$ and assume
     that the group $G$ continuously acts on a $k$-dimensional Hadamard
     manifold $N$. Then, for any $x\in N$, we have
     \begin{align*}
       \dist_N(c(\nu_{G,x}),g c(\nu_{G,x})) \leq \omega_x(1)
      \Big(\frac{k \# S}{\lambda_1(\Gamma(G,S))}\Big)^{1/2}+ \rho \Big(+
      \omega_x (1)\Big(\frac{k \# S}{\lambda_1(\Gamma(G,S))}\Big)^{1/2}\Big)
      \end{align*}for any $g\in G$. There also exists a point
     $z_{x}\in Gx$ such that
     \begin{align*}
      \dist_N (z_{x},gz_{x})\leq \ &\min \Big\{  \omega_x(1)
      \Big(\frac{k \# S}{ \lambda_1(\Gamma(G,S))}\Big)^{1/2}    + \rho
      \Big( +2\omega_x(1)
      \Big(\frac{k \# S}{ \lambda_1(\Gamma(G,S))}\Big)^{1/2}\Big),\\
      & \hspace{1cm}  \omega_x(1)
      \Big(\frac{k \# S}{ \lambda_1(\Gamma(G,S))}\Big)^{1/2}+ 2 \rho
      \Big(  + \omega_x(1)
      \Big(\frac{k \# S}{\lambda_1(\Gamma(G,S))}\Big)^{1/2}\Big)
      \Big\} \tag*{}
      \end{align*}for any $g\in G$.
     \end{thm}

      \section{L\'{e}vy group action}

      In this section, we discuss about a L\'{e}vy group action to
      concrete metric spaces appeared in Section 3.

      A metrizable
    group $G$ is called a \emph{L\'{e}vy group} if it contains an
increasing chain of compact subgroups $G_1\subseteq G_2 \subseteq 
\cdots \subseteq G_n \subseteq  \cdots$ having an everywhere dense union
in $G$ and such that for some right-invariant compatible distance
function $\dist_G$ on $G$ the groups $G_n$, $n\in \mathbb{N}$, equipped
with the Haar measures $\mu_{G_n}$ normalized as $\mu_{G_n}(G_n)=1$ and
    the restrictions of the distance function $\dist_G$, form a L\'{e}vy
    family. See \cite{milgro}, \cite{mil5}, \cite{pestov2}, \cite{pestov4} and references therein for informations about a L\'{e}vy group. 

Let a topological group $G$ acts on a metric space $X$. The action
    is called \emph{bounded} if for any $\varepsilon >0$ there exists a
    neighbourhood $U$ of the identity element $e_{G}\in G$ such that
    $\dist_X(x,gx)<\varepsilon$ for any $g\in U$ and $x\in X$. Note that
 every bounded action is continuous.
 
            \begin{lem}[{cf.~\cite[Theorem 1]{pestov4}}]\label{l3.2}Assume that a metric group $G$ with a right
             invariant distance function $\dist_G$ boundedly acts on a
             metric space $X$. Then, orbit maps $f_x:G\to X$ for all $x\in X$ are
             uniformly equicontinuous.
             \end{lem}

We shall consider an action of a L\'{e}vy group to a metric space $X$
  satisfying the following condition:

  ($\lozenge $): We have $\lim_{n\to
           \infty}\obs_{X}(X_n;-\kappa)=0$ for any $\kappa>0$ and
           any L\'{e}vy family $\{X_n\}_{n=1}^{\infty}$.

 Note that $\mathbb{R}$-trees, doubling spaces, metric graphs, and
 Hadamard manifolds satify the condition ($\lozenge$) (see Section 3).

\begin{conj}
 Any complete Riemannian manifolds satisfy the condition($\lozenge$).
\end{conj}

           Let a topological group $G$ acts on a metric space $X$. We say
    that the
    topological group $G$ acts on $X$ \emph{by uniform isomorphims} if
    for each $g\in G$, the map
    $X\ni x\mapsto gx\in X$ is uniform continuous. The
    action is said to be \emph{uniformly equicontinuous} if for any
    $\varepsilon > 0$ there exists $\delta>0$ such that $\dist_X(gx,gy)<
    \varepsilon$ for every $g\in G$ and $x,y\in X$ with $\dist_X(x,y)<
 \delta$. Given a subset $S\subseteq G$ and $x\in X$, we put $S x:=\{gx
 \mid g\in S\}$.
 \begin{prop}\label{th2}Assume that a L\'{e}vy group $G$ boundedly acts on a metric
 space $X$ having the property ($\lozenge$) by uniform
 isomorphisms. Then for any compact subset $K\subseteq G$ and any
 $\varepsilon >0$, there exists a point $x_{\varepsilon, K}\in X$ such
 that $\diam (K x_{\varepsilon,K})\leq \varepsilon$. 
 \end{prop}

\begin{prop}\label{th3}There are no non-trivial bounded uniformly equicontinuous actions of a L\'{e}vy group on a metric space having the property ($\lozenge$).
 \end{prop}

 \begin{proof}[Proof of Propositions \ref{th2} and \ref{th3}]From the
  definition of $G$, the group $G$ contains an increasing chain of
  compact subgroups $G_1\subseteq G_2 \subseteq \cdots \subseteq G_n
  \subseteq \cdots $ having an everywhere dense union in $G$ such that
  for some right-invariant compatible distance function $\dist_G$ on
  $G$, the sequence $\{(X,\dist_X,\mu_{G_n})\}_{n=1}^{\infty}$ forms a L\'{e}vy
  family. Let $x\in X$ be an arbitrary point.

  We first prove Proposition \ref{th2}. Since $G$ boundedly acts on $X$ and
  $\dist_G$ is right-invarinat, by
  vritue of Lemma \ref{l3.2}, for
  any $\varepsilon >0$ there exists $\delta >0$ such that
  $\dist_X(gy,g'y)<\varepsilon/2$ for any $y\in X$ and $g,g'\in G$ with
  $\dist_G(g,g')\leq \delta$. Take a subset $\{g_1,g_2, \cdots,
  g_N\}\subseteq G$ such that each $g\in K$ is within distance $\delta$
  of the set $\{ g_1,g_2, \cdots, g_N\}$ and all $g_i$ are contained in
  $G_\ell$ for some large $\ell\in \mathbb{N}$. Since the orbit map
  $f_x:G\to X$ is uniformly continuous, by using
  Corollary \ref{c3.2}, the sequence $\{(X,\dist_X, \nu_{G_n,x}) \}_{n=1}^{\infty}$
  is a L\'{e}vy family. From the property
  ($\lozenge$) of the space $X$ the identity maps $\id_n: (X,\dist_X,\nu_{G_n,x})\to X$
  concentrate, that is, $\lim_{n\to \infty}\diam
  (\nu_{G_n,x},1-\kappa)=0$ for any $\kappa >0$. Hence there exist
  $\varepsilon_n >0$ and $x_n \in X_n$ such that $\lim_{n\to
  \infty}\varepsilon_n =0$ and $\lim_{n\to
  \infty}\nu_{G_n,x}(B_X(x_n,\varepsilon_n))=1$. Take $n_0\in \mathbb{N}$ such
  that $n_0\in \mathbb{N}$,
  $\nu_{G_{n_0},x}(B_X(x_{n_0},\varepsilon_{n_0}))>1/2$ and
  $\varepsilon_{n_0}\leq \rho^{(\{g_1, g_2, \cdots, g_N\},X)}(\varepsilon_{n_0})<\varepsilon
  /4$. The same method of the proof of (\ref{s3.1}), we obtain
  \begin{align*}
   \dist_X(x_{n_0}, g_ix_{n_0})\leq \varepsilon_{n_0}+ \rho^{(\{g_1,g_2, \cdots, g_N\},X)}(\varepsilon_{n_0})< \varepsilon /2
   \end{align*}for any $g_i$. For any $g\in K$, choosing $g_i$ with
  $\dist_G(g_i,g)<\delta$, we obtain
\begin{align*}
 \dist_X(x_{n_0}, gx_{n_0})\leq \dist_X(x_{n_0},g_ix_{n_0}) +
 \dist_X(g_i x_{n_0}, g x_{n_0})\leq \frac{\varepsilon}{2}+\frac{\varepsilon}{2}=\varepsilon
 \end{align*}by the definition of $\delta>0$. This completes the proof of Proposition \ref{th2}.

  We next prove Proposition \ref{th3}. Since $\lim_{\eta \to
  0}\rho^{(G,X)}(\eta)=0$, by using Corollary \ref{c3.1}, we
  get
  \begin{align*}
   \diam (G_n x)\leq 2\lim_{\kappa \uparrow 1/2}\diam
   (\nu_{G_n,x},1-\kappa)+ 2\rho^{(G,X)}\big(+\lim_{\kappa \uparrow 1/2}\diam (\nu_{G_n ,x},1-\kappa)\big) \to 0 
   \end{align*}as $n\to \infty$. Since $G_1 x\subseteq G_2
  x\subseteq \cdots \subseteq G_n x \subseteq G_{n+1}x \subseteq \cdots$, we
  therefore obtain $G_n x= \{ x\}$ for any $n\in \mathbb{N}$. This
  completes the proof of Proposition \ref{th3}.
 \end{proof}

  Note that every continuous action of a topological group on a compact
 metric space is bounded. Since a compact metric space has the property
 ($\lozenge$) and a L\'{e}vy group $G$ contains an increasing chain of
 compact subgroups $G_n$ having
 an everywhere dense union, Proposition \ref{th2} includes the fixed
 point theorem (\cite[Theorem 7.1]{milgro}) by Gromov and Milman.

 \begin{ack}\upshape The author would like to express his thanks to
  Professor Takashi Shioya for his valuable suggestions and assistances during the preparation of
  this paper.
  \end{ack}

	\end{document}